\documentclass[11pt,reqno,a4paper]{amsart}
\usepackage{a4wide,fullpage}

\usepackage{graphicx}
\usepackage{amsmath,amsopn,amssymb,amsfonts,amsthm}
\usepackage{mathtools}
\usepackage{color}
\usepackage{enumitem}
\usepackage[framemethod=TikZ]{mdframed}
\usepackage{bbm}
\usepackage{mathrsfs}
\usepackage{booktabs}
\usepackage{caption}
\usepackage{bm}
\usepackage{tensor}
\usepackage{cleveref}
\usepackage{cancel}

\newcommand{\IZ}{{\mathbb Z}}
\renewcommand{\H}{\mathbb{H}}

\newcommand{\N}{\mathbb N}
\newcommand{\C}{\mathbb C}

\theoremstyle{plain}
\newtheorem{thm}{Theorem}[section]
\newtheorem{cor}[thm]{Corollary}
\newtheorem{lem}[thm]{Lemma}

\newtheorem*{rem}{Remark}

\theoremstyle{definition}

\numberwithin{equation}{section}

\def\lp{\left(}
\def\rp{\right)}

\def\k{\kappa}

\def\w{\omega}

\def\th{\theta}
\def\vth{\vartheta}

\def\e{\varepsilon}

\def\k{\kappa}

\def\w{\omega}

\def\th{\theta}
\def\vth{\vartheta}

\def\e{\varepsilon}

\def\del{ \partial}

\newcommand{\re}{{\rm Re}}

\newcommand{\R}{\mathbb R}
\newcommand{\Z}{\mathbb Z}

\newcommand{\Log}{\operatorname{Log}}

\def\wt{\widetilde}

\setlist[itemize]{noitemsep, topsep=0pt}

\allowdisplaybreaks
\crefname{subsection}{subsection}{Subsection}

\makeatletter
\newcommand{\vast}{\bBigg@{3}}
\newcommand{\Vast}{\bBigg@{5}}
\makeatother

\renewcommand{\pmod}[1]{\ \left( \mathrm{mod} \, #1 \right)}
\newcommand{\Pmod}[1]{\ ( \mathrm{mod} \, #1 )}

\newcommand{\Res}{\operatorname{Res}}

\DeclareMathOperator{\lcm}{lcm}
\DeclareMathOperator{\srp}{srp}

\title{Improved asymptotics for moments of reciprocal sums for partitions into distinct parts}
\author[K. Bringmann]{Kathrin Bringmann}
\address{Department of Mathematics and Computer Science\\Division of Mathematics\\University of Cologne\\ Weyertal 86-90 \\ 50931 Cologne \\Germany}
\email{kbringma@math.uni-koeln.de}

\author[B. Kim]{Byungchan Kim}
\address{School of Natural Sciences, Seoul National University of Science and Technology, 232 Gongneung-ro, Nowon-gu, Seoul, 01811, Republic of Korea}
\email{bkim4@seoultech.ac.kr}

\author[E. Kim]{Eunmi Kim}
\address{Institute of Mathematical Sciences, Ewha Womans University,  52 Ewhayeodae-gil, Seodaemun-gu, Seoul 03760, Republic of Korea}
\email{ekim67@ewha.ac.kr; eunmi.kim67@gmail.com}

\makeatletter
\@namedef{subjclassname@2020}{\textup{2020} Mathematics Subject Classification}
\makeatother


\begin{document}

\begin{abstract}
	In this paper we strongly improve asymptotics for $s_1(n)$ (respectively $s_2(n)$) which sums reciprocals (respectively squares of reciprocals) of parts throughout all the partitions of $n$ into distinct parts. The methods required are much more involved than in the case of usual partitions since the generating functions are not modular and also do not posses product expansions.
\end{abstract}

\maketitle

\section{Introduction and statements of results}
A {\it partition} is a non-increasing sequence of positive integers $\lambda_1\ge\lambda_2\ge\dots\ge\lambda_\ell$ such that the {\it parts} $\lambda_j$ sum up to $n$. Let $p(n)$ denote the number of partitions of $n$. Its generating function is
\begin{equation}\label{E:pQop}
	P(q):=\sum_{n\ge0} p(n)q^n = \frac{1}{(q;q)_\infty},
\end{equation}
where, for $a\in\C$, $n\in\N_0\cup\{\infty\}$, we let $(a)_n\!=(a;q)_n\!:=\!\prod_{j=0}^{n-1}(1-aq^j)$. The generating function in \eqref{E:pQop} is (essentially) a modular form. An important question is to determine asymptotic formulas for combinatorial statistics such as $p(n)$. By work of Hardy--Ramanujan \cite{HR}
\begin{equation*}
	p(n) \sim \frac{1}{4\sqrt3 n}e^{\pi\sqrt\frac{2n}{3}} \qquad (n\to\infty).
\end{equation*}
Rademacher improved upon this and showed an exact formula for $p(n)$. A key ingredient of his proof was the modularity of $P(q)$ \cite{RadZ}. To state his result, define the {\it Kloostermann sums}, with the multiplier $\w_{h,k}$, given in \eqref{E:Formula0},
\[
	A_k(n) := \sum_{h\Pmod k^*} \w_{h,k}e^{-\frac{2\pi inh}{k}},
\]
where $h$ only runs over those elements$\Pmod k$ that are coprime to $k$. Then
\begin{equation*}
	p(n) = \frac{2\pi}{(24n-1)^\frac34}\sum_{k\ge1} \frac{A_k(n)}{k}I_\frac32\left(\frac{\pi\sqrt{24n-1}}{6k}\right),
\end{equation*}
where $I_\nu(x)$ denotes the Bessel function of order $\nu$.

A great amount of studies were conducted to understand relations among partitions with some restrictions on the parts. For example, by the famous theorem of Euler, the number of partitions of $n$ into odd parts equals the number of partitions of $n$ into distinct parts. Moreover the first Rogers--Ramanujan identity implies that the number of partitions of $n$ into parts for which the difference between two consecutive parts is at least $2$ equals the number of partitions of $n$ into parts $\equiv \pm 1 \pmod{5}$. These two results concern difference (or gap) conditions of the parts or congruence conditions on parts. Graham \cite{Graham} introduced partitions with conditions on the reciprocal of parts.  We let $\srp (\lambda):= \sum_{j=1}^{\ell(\lambda)} \frac{1}{\lambda_j}$, where $\ell(\lambda)$ is the number of parts of the partition $\lambda$ and we define $\mathcal{D}_n$ to be the set of partitions of $n$ into distinct parts. Graham \cite[Theorem 1]{Graham} proved that there exists a partition $\lambda \in \mathcal{D}_n$ with $\srp(\lambda)=1$ if $n \ge 78$. To understand how the $\srp(\lambda)$ are distributed along $\mathcal{D}_n$, the last two authors \cite{KK} introduced the counting functions\footnote{In \cite{KK}, $s_1(n)$ and $s_2(n)$ were denoted by $s(n)$ and $ss(n)$.} $s_1(n)$ and $s_2(n)$ as

\[
	s_1(n) := \sum_{\lambda \in \mathcal{D}_n}  \srp (\lambda)
	\quad\text{and}\quad
	s_2(n) := \sum_{\lambda \in \mathcal{D}_n} \srp^2(\lambda),
\]
which are the first and the second moments of the reciprocal sums of partitions. To obtain the asymptotic mean and the asymptotic variance of $\srp(\lambda)$ along the set $\mathcal{D}_n$, by employing Wright's Circle Method, they proved \cite[Theorems 1.1 and 1.2]{KK} that, as $n \to \infty$,
\begin{align*}
	s_1(n) &=  \frac{e^{\pi \sqrt{\frac{n}{3}}}}{16\cdot3^\frac14 n^{\frac{3}{4}}}   \log(3n) \left(1+O\left(n^{-\frac{1}{2}}\right) \right),\\
	s_2(n) &=  \frac{e^{\pi \sqrt{\frac{n}{3}}}}{4\cdot 3^\frac14n^{\frac{3}{4}}} \left( \frac{\log^2 (3n)}{16} + \frac{\pi^2}{24}\right) \left(1 +O\left(n^{-\frac{1}{4}} \right)\right).
\end{align*}
In this paper, we drastically improve the asymptotics for $s_1(n)$ and $s_2(n)$ and find asymptotic expansions for them.

We start with $s_1(n)$. Note that the generating function of $s_1(n)$ is related to modular forms -- see Lemma \ref{L:Gen}.
Throughout we let  $\Log$ denote the principal branch of the logarithm.

\begin{thm}\label{T:MainIntro}
	We have\footnote{Note that taking the principal branch is allowed because of \Cref{L:BranchCut}.}, with $X_k(n):=\frac{\pi\sqrt{24n+1}}{6\sqrt2 k}$, as $n\to\infty$,
	\begin{align*}
		s_1 (n)&=\sum_{\substack{0\le h<k\le\left\lfloor\sqrt n\right\rfloor\\\gcd(h,k)=1\\2\nmid k}} \frac{\omega_{h,k}}{\omega_{2h,k}} e^{-\frac{2\pi i n h}k} \vast(\left( \Log\left(\frac{\omega_{h,k}}{2\omega_{2h,k}^2}\right) + \frac{\log(2(24n+1))}{4} \right)\frac{\pi I_1\left(X_k(n)\right)}{ k \sqrt{24n+1}}\\[-.6cm]
		 &\hspace{5cm} + \frac{\pi^2I_2\left(X_k(n)\right)}{4\sqrt{2}k^2(24n+1)}+ \frac{3\sqrt{2}I_0\left(X_k(n)\right)}{24n+1} \vast) + O\left(\sqrt n\right).
	\end{align*}
\end{thm}

From \Cref{T:MainIntro}, we can conclude the following asymptotic main terms.
\begin{cor}\label{C:SAsymp}
	We have, as $n\to\infty$,
	\[
		s_1(n)
		= \frac{e^{\pi\sqrt{\frac n3}} }{16 \cdot 3^{\frac{1}{4}} n^\frac34}   \Bigg( \log(3n) + \left(\frac{\pi}{6}-\frac{9}{\pi} \right)  \frac{\log(3n)}{8\sqrt{3n}}+ \left( \frac{\pi}{4} + \frac{6}{\pi} \right) \frac{1}{\sqrt{3n}} + O\left(\frac{\log (n)}{n}\right) \Bigg).
	\]
\end{cor}
\begin{rem}
	We could determine further terms in the asymptotic expansion of $s_1(n)$.
\end{rem}

 To see how good the asymptotics in Theorem \ref{T:MainIntro} converge to the true value of $s_1(n)$, denote by $s_{1}^{[a]}(n)$  the asymptotic formula in Theorem~\ref{T:MainIntro}. The ratio of $s_1 (n)$ and $s_{1}^{[a]}(n)$ goes to $1$ rapidly as illustrated in Table \ref{tab:comparison}.
\begin{table}[!h]
\begin{tabular}{cccc}
	\toprule
	$n$ & $s_1(n)$ & $s_{1}^{[a]}(n)$ & $\frac{s_1(n)}{s_{1}^{[a]}(n)} -1 $ \\
	\midrule
	100 & 651539.7463 & 651540.1272 & $-5.846709795\cdot 10^{-7}$ \\
	500 & $1.352867042\cdot 10^{15}$ & $1.352867042\cdot10^{15}$ & $9.610589758\cdot 10^{-17}$ \\
	1000 & $1.739106088\cdot 10^{22}$ & $1.739106088\cdot 10^{22}$ & $6.652204586\cdot 10^{-23}$\\
	\bottomrule
\end{tabular}
\medskip
\caption{Values of $s_1(n)$ and $s_{1}^{[a]}(n)$.}\label{tab:comparison}
\end{table}

We next turn to $s_2(n)$. Here the difficulty is that the corresponding $q$-series is not nicely related to modular forms (see \Cref{L:S2Q} and \Cref{L:IntPM}). It would however be very interesting to investigate the spaces in which the generating function of $s_2(n)$ lies (see \Cref{S:Outlook}). To state the asymptotic expansion of $s_2(n)$, we define the following integral of Bessel-type
\begin{equation}\label{E:BesselInt}
	\mathbb I(x) := \int_0^\pi \theta^2 \cos(\theta) e^{x\cos(\theta)} d\theta.
\end{equation}

\begin{thm}\label{T:SS}
	We have
	\begin{align*}
		s_2(n)&= \sum_{\substack{0 \leq h < k \leq N \\ \gcd(h,k) = 1 \\ 2 \nmid k}} \Bigg(\sum_{r=0}^3\gamma_r(h,k;n) I_r\left(X_k(n)\right)+ \log(2(24n+1))\sum_{j=0}^2\delta_j(h,k;n) I_j\left(X_k(n)\right)\\[-.68cm]
		&\hspace{2.4cm} + \log^2(2(24n+1)) \varrho(h,k;n) I_1\left(X_k(n)\right) + \psi(h,k;n) \mathbb I\left(X_k(n)\right)\Bigg)\! + O\left( n^{\frac{3}{2}} \right),
	\end{align*}
	where the constants $\gamma_r$, $\delta_j$, $\varrho$, and $\psi$ are defined in \Cref{S:CM}.
\end{thm}

From \Cref{T:SS}, we can determine the asymptotic main terms.
\begin{cor}\label{C:SSAsymp}
	We have, as $n\to\infty$,
	\begin{multline*}
		s_2(n) =  \frac{e^{\pi\sqrt{\frac n3}} }{4 \cdot 3^{\frac14}n^{\frac34}} \Bigg( \frac{\log^2(3n)}{16} + \frac{\pi^2}{24}  + \left( \frac{\pi}6 - \frac{9}{\pi}\right) \frac{\log^2(3n)}{128 \sqrt{3n}}  +  \left(\frac{\pi}{8} + \frac{3}{\pi} \right)  \frac{\log(3n)}{4\sqrt{3n}}\\
		+ \left( \frac{\pi^3}{144} - \frac{3\pi}{8} - \frac{6}{\pi} -  \frac{7 \zeta(3)}{\pi} \right) \frac{1}{8\sqrt{3n}} + O\left(\frac{\log^2(n)}{n}\right)\Bigg),
	\end{multline*}
	where as usual $\zeta(s)$ denotes the Riemann zeta function.
\end{cor}
\begin{rem}
	Again we could determine further terms in the asymptotic expansion of $s_2(n)$.
\end{rem}
Let $s_{2}^{[a]}(n)$ denote the asymptotic formula in Theorem~\ref{T:SS}. The ratio of $s_2 (n)$ and $s_{2}^{[a]}(n)$ goes to $1$ rapidly as illustrated in Table \ref{tab:comparison2}.
\begin{table}[!h]
\begin{tabular}{cccc}
	\toprule
	$n$ & $s_2(n)$ & $s_{2}^{[a]}(n)$ & $\frac{s_2(n)}{s_{2}^{[a]}(n)} -1 $ \\
	\midrule
	100 & $1.123454714 \cdot 10^6$ & $1.123456875 \cdot 10^6$ & $-1.922856336 \cdot 10^{-6}$ \\
	500 & $2.787849982 \cdot 10^{15}$ & $2.787849982 \cdot 10^{15}$ & $7.085287046 \cdot 10^{-16}$ \\
	1000 & $3.848623818 \cdot 10^{22}$ & $3.848623818 \cdot 10^{22}$ & $2.190018609 \cdot 10^{-22}$\\	 
	\bottomrule
\end{tabular}
\medskip
\caption{Values of $s_2(n)$ and its asymptotic formula $s_{2}^{[a]}(n)$.}\label{tab:comparison2}
\end{table}

The paper is organized as follows: In \Cref{S:Preliminaries}, we recall basic facts on the partition multiplier, Bessel functions, Bernoulli polynomials, the polylogarithm, the Euler--Maclaurin summation formula, Farey fractions, and the approximation of certain integrals using Bessel functions. In \Cref{S:Modularity}, we rewrite the corresponding generating functions and determine their asymptotic behavior. \Cref{S:CM} is devoted to the proof of our main theorems using the Circle Method and our previous approximations. In \Cref{S:Proofs}, we determine the main term contributions of our partition statistics. We end in \Cref{S:Outlook} with some open questions.

\section*{Acknowledgments}
The authors thank Caner Nazaroglu for helping with numeric verification of our results.
The first author has received funding from the European Research Council (ERC) under the European Union's Horizon 2020 research and innovation programme (grant agreement No. 101001179). The second author was supported by the Basic Science Research Program through the National Research Foundation of Korea (NRF) funded by the Ministry of Science and ICT (NRF--2022R1F1A1063530). The third author was supported by the Basic Science Research Program through the National Research Foundation of Korea (NRF) funded by the Ministry of Education (RS--2023--00244423, NRF--2019R1A6A1A11051177).

\section{Preliminaries}\label{S:Preliminaries}

\subsection{The partition multiplier}\label{S:Multiplier}

Let $z\in\C$ with $\re(z)>0$ and let $h,k\in\N_0$ with $0\le h< k$, $\gcd(h,k)=1$.
Let $\omega_{h,k}$ be defined through
\begin{equation}\label{E:Pt}
	P(q) = \omega_{h,k}\sqrt ze^{\frac{\pi}{12k}\left(\frac1z-z\right)} P(q_1),
\end{equation}
where $q:=e^{\frac{2\pi i}k(h+iz)}$ and $q_1:=e^{\frac{2\pi i}k(h'+\frac iz)}$ with $hh'\equiv-1\Pmod k$. Note that we have
\[
	\omega_{h,k}:=e^{\pi i s(h,k)},
\]
where the {\it Dedekind sum} $s(h,k)$ is defined as
\[
	s(h,k):=\sum_{\mu\pmod{k}}\left(\!\!\left(\frac{\mu}{k}\right)\!\!\right)\left(\!\!\!\left(\frac{h\mu}{k}\right)\!\!\!\right)
\]
with
\[
	\left(\!\left(x\right)\!\right):=
	\begin{cases}
		x-\lfloor x\rfloor-\frac{1}{2}&\text{if }x\in\R\setminus\Z,\\
		0&\text{if }x\in\Z.
	\end{cases}
\]
We require the following representation (see \cite[equation (5.2.4)]{A})
\begin{equation}\label{E:Formula0}
	\omega_{h,k} =
	\begin{cases}
		\left(\frac{-k}{h}\right)e^{-\pi i\left(\frac14(2-hk-h)+\frac{1}{12}\left(k-\frac1k\right)\left(2h-h'+h^2h'\right)\right)} & \text{if $2\nmid h$},\\
		\left(\frac{-h}{k}\right)e^{-\pi i\left(\frac14(k-1)+\frac{1}{12}\left(k-\frac1k\right)\left(2h-h'+h^2h'\right)\right)} & \text{if $2\mid h$},
	\end{cases}
\end{equation}
where $(\frac\cdot\cdot)$ denotes the Kronecker symbol.

To avoid the branch cut of the logarithm, we need the following fact.

\begin{lem}\label{L:BranchCut}
	We have\footnote{Note that here and throughout we choose $h'$ to satisfy $8\mid h'$.}, for $2\nmid k$,
	\begin{equation*}
		\frac{\omega_{h,k}}{\omega_{2h,k}^2} \not\in \R^-.
	\end{equation*}
\end{lem}
\begin{proof}
	First note that a direct calculation, using \eqref{E:Formula0}, yields
	\begin{equation*}
		\frac{\omega_{h,k}}{\omega_{2h,k}^2} = \left(\frac{-h}{k}\right) e^{\frac{\pi i}{4k}\left(h\left(k^2-1\right)+ k^2(1-k)\right)}.
	\end{equation*}
	From this it is then not hard to see the claim.
\end{proof}

\subsection{Bessel functions and their growth}\label{SubsBessel}
The \emph{$I$-Bessel function} is defined by\footnote{For some properties of the modified Bessel functions, see \cite[Subsections 3.7, 6.22, and 7.23]{W}.}
\begin{equation*}
	I_\nu (z) = \frac{\left( \frac{z}{2} \right)^\nu}{2 \pi i}  \int_{\mathcal D} w^{-\nu-1} \exp \left( w + \frac{z^2}{4w} \right) dw \quad \text{ for } z\in\C\setminus\R_0^-, \nu \in \C,
\end{equation*}
where $\mathcal D$ is Hankel's contour that starts at $-\infty$ below the real axis, circles around the negative real axis, and then goes back to $-\infty$ above the real line. We also require the integral representation, for $x\in\R^+$,
\begin{equation}\label{Iint}
	I_\nu(x)=\frac1\pi \int_0^\pi \cos(\nu\theta) e^{x\cos(\theta)}d\theta-\frac{\sin(\pi\nu)}{\pi} \int_0^\infty e^{-x\cosh(t)-\nu t}dt.
\end{equation}
Note that, for $n\in\N$,
\begin{equation}\label{I_sym}
	I_{-n}(z) = I_n(z).
\end{equation}

The \emph{$K$-Bessel function} is for $z\in\C\setminus \R^-_0$ defined by
\begin{align*}
	K_\nu(z) &:= \frac{\pi}{2} \frac{I_{-\nu}(z) - I_\nu(z)}{\sin(\pi \nu)}, \quad \nu \in \C\setminus \Z, \qquad
	&K_n(z) &:= \lim_{\nu \to n} K_\nu(z), \quad  n \in  \Z.
\end{align*}
We have (see 10.38.3 of \cite{Ni} or \cite[equation (3.8)]{NR}) for $n\in\N_0$ 
\begin{equation}\label{E:IK}
	(-1)^n \left[ \frac{\partial}{\partial \nu} I_{\nu}(z) \right]_{\nu = \pm n} = - K_n(z) \pm  \frac{n!}{2\left(\frac{z}{2}\right)^n} \sum_{k=0}^{n-1} (-1)^k \frac{\left(\frac{z}{2}\right)^k I_k(z)}{(n-k)k!}.
\end{equation}
We also require the asymptotics of the $I$-Bessel function and the $K$-Bessel function as $x \to \infty$,
\begin{equation}\label{Bbound}
	I_{\nu} (x) = \frac{e^{x}}{\sqrt{2 \pi x}} \left(1 + \frac{1-4\nu^2}{8x}+O\left(\frac1{x^2}\right)\right), \qquad
	K_{\nu} (x) = \sqrt{\frac{\pi}{2 x}}e^{-x} \left(1 + O\left(\frac1x\right)\right).
\end{equation}

We next determine the asymptotic behavior of our Bessel function integral defined in \eqref{E:BesselInt}. A direct calculation, using the saddle point method, gives the following estimate.

\begin{lem}\label{L:BesselInt}
	We have, as $x\to\infty$,
	\begin{equation*}
		\mathbb I(x) = \sqrt{\frac\pi2} \frac{e^x}{x^\frac32}\left(1+O\left(\frac1x\right)\right).
	\end{equation*}
\end{lem}

\subsection{Bernoulli polynomials}
Define the {\it Bernoulli polynomials} via their generating function
\[
	\frac{t e^{xt}}{e^t-1} =: \sum_{n \geq 0} B_n(x) \frac{t^n}{n!}.
\]
We also use the Fourier series of the Bernoulli periodic polynomials $\widetilde{B}_n(x) := B_n\lp x - \lfloor x \rfloor \rp$,
\begin{equation}\label{E:TildeBEll}
	\widetilde{B}_\ell(x) = -\frac{2\cdot\ell !}{(2\pi)^\ell} \sum_{n\ge1} \frac{\cos\left(2\pi n x-\frac{\pi\ell}2\right)}{n^\ell}.
\end{equation}
We further require
\begin{equation}\label{E:Bvan}
	B_n(1-x) = (-1)^n B_n(x).
\end{equation}

\subsection{Polylogarithms}
Define for $\ell\in\N$ and $w\in\C$ with $|w|<1$ the {\it polylogarithm}\footnote{For general facts on polylogarithms, see \cite{Z2}.}
\begin{equation*}
	\operatorname{Li}_\ell(w) := \sum_{n\ge1} \frac{w^n}{n^\ell}.
\end{equation*}
In particular, $\operatorname{Li}_1(w)=-\Log(1-w)$ and, for $\ell\ge2$, $\operatorname{Li}_\ell(1)=\zeta(\ell)$. We have the identity
\begin{equation}\label{E:DiffLi}
	\frac d{dw} \operatorname{Li}_\ell(w) = \frac1w\operatorname{Li}_{\ell-1}(w) \quad \text{for $\ell\ge2$}.
\end{equation}
This implies in particular
\begin{equation}\label{E:LiInt}
	\int_0^\infty \operatorname{Li}_3\left(e^{-\pi x}\right) dx = \frac{\pi^3}{90}.
\end{equation}
We also require the reflection formula
\begin{equation}\label{E:LRel}
	\operatorname{Li}_2(w)+\operatorname{Li}_2\left(\frac{1}{w}\right)=-\frac{\pi^2}{6}-\frac{\Log^2(-w)}2.
\end{equation}

\subsection{The Euler--Maclaurin summation formula}
We recall the following formula, which is slightly modified from an exact formula given in \cite[p. 13]{Z}, as derived for example in \cite{BJM}.
\begin{lem}\label{lem:32}
	Let $a \geq 0$, $M \in \mathbb{N}$, $f$ a holomorphic function on a domain containing $D_\theta:=\{re^{i \alpha}:|\alpha|\leq\theta\}$ for some $-\frac\pi2<\theta<\frac\pi2$, and suppose that $f^{(n)} (w) \ll |w|^{-1-\e}$ for $n \in \{0,1,\ldots,M\}$ as $|w| \to \infty$ in that domain (with $\e > 0$). Then for any $w \in D_\th\setminus\{0\}$ we have
	\begin{align*}
		\sum_{m\ge0} f((m+a)w)
		\!=\! \frac{1}{w} \int_{0}^{\infty} f(x) dx
		\!-\!\! \sum_{n=0}^{M-1} \! \frac{B_{n+1} (a)}{(n+1)!}  f^{(n)}(0) w^n
		\!-\! \frac{(-w)^{M}}{M!} \! \int_{0}^\infty f^{(M)}(wx) \wt{B}_M(x-a) dx.
	\end{align*}
\end{lem}

\subsection{Farey fractions}\label{S:Farey}
Let $0\le h<k\le N$ with $\gcd(h,k)=1$, and $z=\frac kn -ik\Phi$ with $-\vth_{h,k}'\le\Phi\le\vth_{h,k}''$. Here
\[
	\vth_{h,k}' := \frac{1}{k(k_1+k)},\qquad \vth_{h,k}'' := \frac{1}{k(k_2+k)},
\]
where $\frac{h_1}{k_1}<\frac hk<\frac{h_2}{k_2}$ are adjacent Farey fractions in the Farey sequence of order $N:=\lfloor\sqrt n\rfloor$. From the theory of Farey fractions it is well-known that
\begin{equation}\label{Fbound}
	\frac1{k+k_j} \le \frac1{N+1} \quad \text{for } j\in\{1, 2\}.
\end{equation}
Moreover we have
\begin{equation}\label{zbound}
	\re(z)=\frac{k}{n},\quad\re \left(\frac{1}{z}\right) \ge \frac{k}{2}, \quad |z| \ll \frac1{\sqrt n}, \quad |z| \ge \frac kn.
\end{equation}

\subsection{Integral approximations}\label{S:Integrals}
 In this section we give approximations of certain integrals that are required in \Cref{S:Proofs}. Following \cite L we can show the following. 
\begin{lem}\label{L:BesselBound}
	Suppose that $k\in\N$, $s\in\R^+$, and let $\vth_1,\vth_2,a,b\in\R^+$ satisfy $k\ll\sqrt{n}$, $a\asymp\frac nk$, $b\ll\frac1k$, and $k\vth_1$, $k\vth_2\asymp\frac{1}{\sqrt{n}}$. Then we have
	\[
	\int_{\frac kn-ik\vth_1}^{\frac kn+ik\vth_2} Z^{s}e^{aZ+\frac bZ}dZ = 2\pi i\left(\frac ba\right)^\frac{s+1}{2}I_{-s-1}\left(2\sqrt{ab}\right) + f_{a,b}(s),
	\]
	where for $(\ell, j) \in \{(1,0), (1,1), (2,0) \}$, we have
	\begin{align*}
		f_{a,b}(s) &\ll n^{-\frac{s+1}2}, \quad f_{a,b}^{(\ell)}(j)\ll\frac{\log^\ell(n)}{n^\frac{j+1}{2}}.
	\end{align*}
\end{lem}

\section{Principal parts}\label{S:Modularity}
In this section, we determine the principal parts of the generating functions of $s_1(n)$ and $s_2(n)$. As in \Cref{S:Multiplier}, $q=e^{\frac{2\pi i}{k}(h+iz)}$, where we assume the notation from \Cref{S:Farey}. In particular $k|z|\ll 1$, as $z\to0$.

\subsection{$s_1(n)$}
In this subsection, we approximate the generating function of $s_1(n)$, which is defined as (see Subsection 4.1 of \cite{KK})
\begin{equation*}
		S_1(q) := \sum_{n\ge1} s_1(n) q^n=(-q;q)_\infty\sum_{r\ge1}\frac{q^r}{r\left(1+q^r\right)}.
\end{equation*}
It is not hard to show the following relation to $P(q)$.

\begin{lem}\label{L:Gen}
	We have
	\[
		S_1(q) = \frac{P(q)}{P(q^2)} \Log\left( \frac{P(q)}{P^2(q^2)} \right).
	\]
\end{lem}

We next determine the principal part, i.e., the part that grows as $z\to0$, of $S_1(q)$. For this, we define, for $k$ odd,
\begin{equation*}
	\mathcal P_{h,k}(z):= \frac{\omega_{h,k}}{\sqrt2\omega_{2h,k}} e^{\frac{\pi}{24kz}+\frac{\pi z}{12k}} \left(\Log\left(\frac{\omega_{h,k}}{2\omega_{2h,k}^{2}}\right)- \frac{\Log(z)}2 + \frac{\pi z}{4k}\right).
\end{equation*}
Using \Cref{L:Gen} and \eqref{E:Pt}, we obtain by a direct calculation.

\begin{lem}\label{L:Main}
	Assume the notation above.
	\begin{enumerate}[wide, labelwidth=!, itemindent=!, labelindent=0pt,label=\rm(\arabic*)]
		\item\label{I:Main1} If $2\mid k$, then, as $z\to0$,
		\begin{equation*}
			S_1(q) \ll \left(\Log(z) + \frac1{kz}\right) e^{-\frac\pi{12kz}}.
		\end{equation*}
		\item\label{I:Main2} If $2\nmid k$, then, as $z\to0$,
		\begin{equation*}
			S_1(q) = \mathcal P_{h,k}(z) + \mathcal E_{h,k}(z), \quad \text{where} \quad \mathcal E_{h,k}(z) \ll \Log(z)e^{-\frac{23\pi}{24kz}}.
		\end{equation*}
	\end{enumerate}
\end{lem}

\subsection{$s_2(n)$}
In this subsection, we approximate the generating function of $s_2(n)$ which is defined as (see Subsection 4.1 of \cite{KK})
\begin{equation*}
	S_2(q) := \sum_{n\ge0} s_2(n) q^n = (-q;q)_\infty \left( \frac{S_{1}^{2}(q)}{(-q;q)_\infty^2} + \sum_{r\ge1}\frac{q^r}{r^2\left(1+q^r\right)^2} \right).
\end{equation*}
To bound the factor outside, we use \eqref{E:Pt} to obtain, if $2\mid k$,
\begin{equation}\label{P2}
	(-q;q)_\infty \ll e^{-\frac{\pi}{12kz}}.
\end{equation}
Moreover, if $2\nmid k$, then we have
\begin{equation}\label{Pn2}
	(-q;q)_\infty = \frac{\omega_{h,k}}{\sqrt2 \omega_{2h,k}}e^{\frac{\pi}{24kz}+\frac{\pi z}{12k}}+O\left(e^{-\frac{23\pi}{24kz}}\right).
\end{equation}
\Cref{L:Main} then directly implies the following.

\begin{lem}\label{L:CircleMethod}
	Assume the notation above.
	\begin{enumerate}[wide,labelwidth=!, itemindent=!, labelindent=0pt, label=\rm(\arabic*)]
		\item If $2\mid k$, then, as $z\to0$,
		\begin{equation*}
			S_1^2(q) \ll \left(\Log^2(z) + \frac1{k^2z^2}\right) e^{-\frac\pi{6kz}}.
		\end{equation*}
		\item If $2\nmid k$, then, as $z\to0$,
		\begin{equation*}
			S_1^2(q) = \mathcal{P}_{h,k}^2(z) + O\left(\Log^2(z)e^{-\frac{11\pi}{12kz}}\right).
		\end{equation*}
	\end{enumerate}
\end{lem}

We next turn to
\begin{equation*}
	L(q):=\sum_{r\ge1}\frac{q^r}{r^2\left(1+q^r\right)^2}.
\end{equation*}
The next lemma is a first step in determining the growth of $L(q)$, as $z\to0$. Here $\delta_S:=1$ if a statement $S$ holds and $\delta_S:=0$ otherwise. Also $\zeta_j:=e^\frac{2\pi i}j$ for $j\in\N$. Moreover let, for $x\in\R$,
\begin{equation*}
	\{x\} := x-\lfloor x\rfloor.
\end{equation*}
\begin{lem}\label{L:S2Q}
	Assume the notation above. We have, with $\kappa:=\lcm(k,2)$,
	\begin{multline*}
		L(q) =- \frac{\pi^2 \delta_{2\mid k}}{24 k^2 z^2} + \frac{\kappa}2 \sum_{\ell=1}^\kappa (-1)^\ell B_2\left(\frac\ell\kappa\right) \operatorname{Li}_2\left(\zeta_k^{h\ell}\right) + \frac{2\pi^2i\kappa^2}{3k} z \sum_{\ell=1}^\kappa (-1)^{\ell} \left\{\frac{h\ell}k\right\} B_3\left(\frac\ell\kappa\right)\\
		-\frac{\k^3}{2\pi k^2} \sum_{\ell=1}^{\k}(-1)^\ell \sum_{r \geq 1} \frac{1}{r^3} \frac{\partial}{\partial z} \left(z^3 \int_{0}^{\infty} \sum_{\pm}\frac{\sin\left(2\pi r\left(x\pm \frac{\ell}{\k}\right)\right)}{\zeta_k^{\pm h\ell}e^{2\pi\frac{\k}{k}zx}-1}dx\right).
	\end{multline*}
\end{lem}
\begin{proof}
	We have
	\begin{align*}
		L(q) = \sum_{r\ge1} \frac{e^{\frac{2\pi i}k(h+iz)r}}{r^2\left(1+e^{\frac{2\pi i}k(h+iz)r}\right)^2} = \sum_{r\ge1} \frac{\zeta_k^{hr}e^{-\frac{2\pi rz}k}}{r^2\left(1+\zeta_k^{hr}e^{-\frac{2\pi rz}k}\right)^2} = \frac k{2\pi} \frac\partial{\partial z} \sum_{r\ge1} \frac1{r^3\left(1+\zeta_k^{hr}e^{-\frac{2\pi rz}k}\right)}.
	\end{align*}
	Now let
	\begin{equation*}
		G_{h,k}(z) := \sum_{r\ge1} \frac1{r^3\left(1+\zeta_k^{hr}e^{-\frac{2\pi rz}k}\right)} = \sum_{\substack{r\ge 1\\m \geq 0}}  \frac{\left(-\zeta_k^{hr}\right)^me^{-\frac{2\pi rmz}k}}{r^3} = \sum_{m\ge 0} (-1)^m \operatorname{Li}_3\left(\zeta_k^{hm}e^{-\frac{2\pi mz}k}\right).
	\end{equation*}
	Making the change of variables $m\mapsto\kappa m+\ell$ with $m\in\N_0$ and $0\le\ell\le\kappa-1$, we obtain
	\begin{equation}\label{E:Ghk}
		G_{h,k}(z) \!=\! \frac12\sum_{\ell=1}^{\kappa} (-1)^\ell \sum_{m\ge0} \left(f_{h,k,\ell}\!\left(\left(m+\frac\ell \k\right)\!\frac{\k}{k}z\right)\!+\!f_{h,k,\kappa-\ell}\!\left(\left(m+1-\frac\ell \k\right)\!\frac{\k}{k}z\right)\right) \!+\! \frac{\zeta(3)}2,
	\end{equation}
	where
	\begin{equation*}
		f_{h,k,\ell}(x) := \operatorname{Li}_3\left(\zeta_k^{h\ell}e^{-2\pi x}\right).
	\end{equation*}

	Now \Cref{lem:32} gives
	\begin{multline}\label{E:FSum}
		\sum_{m\ge0} f_{h,k,\ell}\left(\left(m+\frac\ell {\kappa}\right)\frac{\kappa}{k}z\right) = \frac{k}{\kappa z} \int_{0}^{\infty} \operatorname{Li}_{3}\left(\zeta_k^{h\ell}e^{-2\pi x}\right)dx - \sum_{r=0}^{1} \frac{B_{r+1}\left(\frac{\ell}{\kappa}\right)}{(r+1)!} f_{h,k,\ell}^{(r)}(0)\left(\frac{\kappa}{k}z\right)^r\\
		- \frac{\kappa^2}{2k^2} z^2 \int_{0}^{\infty} f_{h,k,\ell}^{(2)}\left(\frac{\kappa}{k}zx\right)\widetilde{B}_{2}\left(x-\frac{\ell}{\kappa}\right)dx.
	\end{multline}

	The first integral converges absolutely by \eqref{E:LiInt}. We then compute
	\begin{multline}\label{E:IntSum}
		\sum_{\ell=1}^{\kappa} (-1)^{\ell} \int_{0}^{\infty}  \left(\operatorname{Li}_{3}\left(\zeta_k^{h\ell}e^{-2\pi x}\right) + \operatorname{Li}_3\left(\zeta_k^{h(\kappa-\ell)}e^{-2\pi x}\right)\right) dx\\
		= \sum_{\ell=1}^{\kappa} (-1)^\ell \int_{0}^{\infty} \sum_{\pm} \sum_{r\ge1} \tfrac{\left(\zeta_k^{\pm h\ell}e^{-2\pi x}\right)^r}{r^3}dx.
	\end{multline}
	Now
	\begin{equation*}
		\sum_{\ell=1}^{\kappa} (-1)^\ell \zeta_{k}^{\pm hr\ell}=	\sum_{\ell=1}^{\kappa} \zeta_{\kappa}^{\pm\left(\frac{\kappa}{k}hr+\frac{\kappa}{2}\right)\ell}=
		\begin{cases}
			\kappa & \text{if} \ \frac{\kappa}{k}hr+\frac{\kappa}{2}\equiv0\pmod{\kappa},\\
			0 & \text{otherwise}.
		\end{cases}
	\end{equation*}
	From this, one can see that for $2\nmid k$, (\ref{E:IntSum}) vanishes. If $2\mid k$, then \eqref{E:IntSum} is non-vanishing if and only if $\frac k2\mid r$ and $2\nmid\frac r{\frac k2}$. Using \eqref{E:LiInt}, multiplying by $\frac{k}{2\pi}$, and differentiating with respect to $z$ gives the first term in the lemma.

	Next note that we can ignore the contribution  from the constant term in \eqref{E:FSum} as well as the term $\frac{\zeta(3)}{2}$ in \eqref{E:Ghk}, as this goes away upon differentiating.

	Using \eqref{E:DiffLi}, the contribution from the $z$-term gives the second term in the lemma. Finally the last integral term in \eqref{E:FSum} equals, using \eqref{E:TildeBEll},
	\begin{equation*}
		-\frac{\kappa^2}{2\pi^2k^2} z^2 \sum_{r\ge1}\frac1{r^2} \int_{0}^{\infty} f_{h,k,\ell}^{(2)}\left(\frac\kappa k zx\right) \cos\left(2\pi r\left(x-\frac{\ell}{\kappa}\right)\right)dx.
	\end{equation*}
	Computing
	\begin{equation*}
		f_{h,k,\ell}^{(2)} (w) = -4\pi^2 \Log\left(1-\zeta_k^{h\ell}e^{-2\pi w}\right),
	\end{equation*}
	this contributes to $G_{h,k}(z)$ as
	\begin{equation}\label{ContG}
		\frac{\kappa^2}{k^2}z^2 \sum_{\ell=1}^{\kappa} (-1)^\ell \sum_{r\ge1} \frac1{r^2} \int_{0}^{\infty} \sum_\pm\Log\left(1-\zeta_{k}^{\pm h\ell}e^{-2\pi\frac\kappa kzx}\right) \cos\left(2\pi r\left(x\mp\frac{\ell}{\kappa}\right)\right)dx.
	\end{equation}
	We compute, using integration by parts, 
	\begin{multline*}
		\int_{0}^{\infty} \Log\left(1-\zeta_{k}^{\pm h\ell}e^{-2\pi\frac\kappa kzx}\right) \cos\left(2\pi r\left(x\mp\frac{\ell}{\kappa}\right)\right) dx\\
		=\frac{\pm\delta_{k\nmid\ell}}{2\pi r} \Log\left(1-\zeta_{k}^{\pm h\ell}\right) \sin\left(\frac{2\pi r\ell}{\kappa}\right) -\frac{\kappa z}{kr} \int_{0}^{\infty} \frac{\sin\left(2\pi r\left(x\mp\frac{\ell}{\kappa}\right)\right)}{\zeta_{k}^{\mp h\ell}e^{2\pi\frac\kappa kzx}-1}dx.
	\end{multline*}
	The final term immediately yields the final term in the lemma by multiplying by $\frac k{2\pi}$, and differentiating with respect to $z$. The first term contributes to \eqref{ContG} as
	\begin{equation}\label{E:LogTerm}
		\frac{\kappa^2}{2 \pi k^2} z^2 \sum_{\substack{\ell=1\\k\nmid\ell}}^{\kappa} (-1)^{\ell} \left( \Log\left(1-\zeta_{k}^{h\ell}\right) - \Log\left(1-\zeta_{k}^{-h\ell}\right) \right) \sum_{r\ge1} \frac{\sin\left(\frac{2\pi r\ell}{\kappa}\right)}{r^3}.
	\end{equation}
	We next simplify
	\begin{align*}
		\Log\left(1-\zeta_k^{h\ell}\right) - \Log\left(1-\zeta_k^{-h\ell}\right) &= i\operatorname{Arg}\left(1-\zeta_k^{h\ell}\right) - i\operatorname{Arg}\Big(-\zeta_k^{-h\ell}\left(1-\zeta_k^{h\ell}\right)\Big)\\
		&= - 2\pi i\left( \frac{1}{2} -  \left\{ \frac{h \ell}{k} \right\} \right),
	\end{align*}
	using the easily verified identity ($z_1=-e^{-i\alpha}$, $z_2=1-e^{i\alpha}$ with $0<\alpha<2\pi$)
	\begin{equation*}
		\operatorname{Arg}(z_1z_2) - \operatorname{Arg}(z_1) - \operatorname{Arg}(z_2) = 2\arctan\left(\frac{\sin(\alpha)}{1-\cos(\alpha)}\right) - \pi + \alpha = 0.
	\end{equation*}
	Thus, by \eqref{E:TildeBEll}, \eqref{E:LogTerm} becomes
	\begin{equation}\label{E:FirstTerm}
		- \frac{2 \pi^3 i\kappa^2 }{3 k^2} z^2 \sum_{\substack{\ell=1\\k\nmid\ell}}^\kappa (-1)^\ell \left( \frac{1}{2} -  \left\{ \frac{h \ell}{k} \right\} \right) B_3\left( \frac{\ell}{\kappa}\right)
		=  \frac{2 \pi^3 i \kappa^2}{3 k^2} z^2 \sum_{\substack{\ell=1\\k\nmid\ell}}^\kappa (-1)^{\ell} \left\{ \frac{h \ell}{k} \right\}  B_3\left( \frac{\ell}{\kappa}\right),
	\end{equation}
	where for the last equality above, we use
	\[
		\sum_{\substack{\ell = 1\\ k \nmid \ell}}^{\kappa} (-1)^\ell B_3 \left( \frac{\ell}{\kappa}\right) = 0.
	\]
	Next note that as $\{x\}\!=\!0$ for $x\in\Z$, the condition $k\nmid\ell$ in \eqref{E:FirstTerm} may be dropped. The third term in the lemma is given by multiplying by $\frac k{2\pi}$ and differentiating with respect to $z$.
\end{proof}

We next explicitly evaluate the integral occurring in \Cref{L:S2Q}.
\begin{lem}\label{L:IntPM}
	We have
	\[
		\int_0^\infty \sum_\pm \frac{\sin\left(2\pi r\left(x\pm\frac\ell \kappa \right)\right)}{\zeta_k^{\pm h\ell}e^{2\pi\frac\kappa kzx}-1} dx
		=  -\frac{\cos\left(\frac{2\pi r\ell}\kappa\right)}{2\pi r}  +  \frac{k}{2\kappa z}e^{-\frac{2\pi r}{\kappa} \left( \frac{\lambda}z + i \ell \right)} + \frac{k \cosh \left(\frac{2\pi r}{\kappa} \left( \frac{\lambda}z + i \ell \right)\right)}{ \kappa z\left(e^{\frac{2\pi r k}{\kappa z}}-1\right)},
	\]
	where $\lambda$ is uniquely determined as $0\le\lambda< k$ such that $\lambda\equiv h\ell \pmod{k}$.
\end{lem}
\begin{proof}
	By analytic continuation we may assume that $z\in\R$. We define
	\begin{equation*}
		F_{r,\ell,h,k} := \operatorname{CT}_w F_{r,\ell,h,k}(w),  
	\end{equation*}
	where $\operatorname{CT}_w F(w)$ denotes the constant term of the Laurent expansion of a meromorphic function $F(w)$, and where\footnote{Note that our original function has $x\mapsto zx$ introducing an extra factor of $\frac{1}{z}$.}
	\begin{equation*}
		F_{r,\ell,h,k}(w) := \int_0^\infty \sum_\pm \frac{\sin\left(2\pi r\left(\frac xz\pm\frac\ell \kappa \right)\right)}{\zeta_k^{\pm h\ell}e^{2\pi\frac\kappa kx}-1} e^{-wx} dx.
	\end{equation*}
	We now use
	\begin{equation*}
		\frac2{e^{x}-1} = \coth\left(\frac{x}{2}\right) -1.
	\end{equation*}
	and split, for $\re(w)>0$,
	\begin{equation*}
		F_{r,\ell,h,k}(w) = -F_{r,\ell,h,k}^{[1]}(w) + F_{r,\ell,h,k}^{[2]}(w),
	\end{equation*}
	where
	\begin{align*}
		F_{r,\ell,h,k}^{[1]}(w) &:= \frac12\int_0^\infty \sum_\pm \sin\left(2\pi r\left(\frac xz\pm\frac\ell\kappa\right)\right) e^{-wx} dx,\\
		F_{r,\ell,h,k}^{[2]}(w) &:= \frac12\int_0^\infty \sum_\pm \coth\left(\frac{\pi\kappa}kx\pm\frac{\pi ih\ell}k\right) \sin\left(2\pi r\left(\frac xz\pm\frac\ell\kappa\right)\right) e^{-wx} dx.
	\end{align*}
	Noting that $\coth(x+\pi i\ell) = \coth(x)$ for $\ell\in\Z$, we have, with $\lambda$ defined as in the lemma,
	\begin{equation*}
		\coth\left(\frac{\pi\kappa}{k}x\pm\frac{\pi i h\ell}{k}\right)=\coth\left(\frac{\pi\kappa}{k}x\pm\frac{\pi i\lambda}{k}\right).
	\end{equation*}
 	We next note that, using integration by parts
	\begin{equation*}
		\int_0^\infty \sin(Ax+B) e^{-wx} dx = \frac{\sin(B)w+A\cos(B)}{A^2+w^2} \to \frac{\cos(B)}A,
	\end{equation*}
	as $w\to0$. Thus
	\begin{equation*}
		F_{r,\ell,h,k}^{[1]}(w) \to \frac12 \sum_\pm \frac{\cos\left(\pm\frac{2\pi r\ell}\kappa\right)}{\frac{2\pi r}z} = \frac z{2\pi r} \cos\left(\frac{2\pi r\ell}\kappa\right).
	\end{equation*}

	For $\operatorname{F}_{r,\ell,h,k}^{[2]}(w)$, we use that
	\begin{equation*}
		\coth(w) = \frac1w + 2w \sum_{m\ge1}\frac1{w^2+\pi^2m^2}
	\end{equation*}
	to obtain
	\begin{multline*}
		F_{r,\ell,h,k}^{[2]}(w) = \frac12 \sum_\pm \int_{0}^{\infty} \left(\frac{1}{\frac{\pi\kappa}kx\pm\frac{\pi i\lambda}{k}}\right.\\
		\left.+ 2\left(\frac{\pi\kappa}kx\pm\frac{\pi i\lambda}{k}\right)\sum_{m \geq 1 } \frac{1}{\left(\frac{\pi\kappa}kx\pm\frac{\pi i\lambda}{k}\right)^{2} + \pi^2m^2}\right) \sin\left(2\pi r\left(\frac{x}{z}\pm\frac{\ell}{\kappa}\right)\right)e^{-wx}dx.
	\end{multline*}
	By Lebesque's Theorem on dominated convergence we may pull the limit $w\to0$ inside. We need to therefore understand ($m\in\N_0$)
	\begin{equation*}
		\frac1{2i}\sum_\pm \int_{0}^{\infty} \frac{\left(\frac{\pi\kappa}kx\pm\frac{\pi i\lambda}k \right)\left(e^{2\pi ir\left(\frac{x}{z}\pm\frac{\ell}\kappa\right)}-e^{-2\pi ir\left(\frac{x}{z}\pm\frac{\ell}\kappa\right)}\right)}{\left(\frac{\pi\kappa}kx\pm\frac{\pi i\lambda}k \right)^2 + \pi^2m^2}dx = \frac1{2i}\sum_\pm \int_\R \frac{\left(\frac{\pi\kappa}kx\pm\frac{\pi i\lambda}k \right)e^{2\pi ir\left(\frac xz\pm\frac\ell{\kappa}\right)}}{\left(\frac{\pi\kappa}kx\pm\frac{\pi i \lambda}k \right)^2+\pi^2m^2} dx,
	\end{equation*}
	making the change of variables $x\mapsto-x$ in the second term. We now use the Residue Theorem and deform the path of integration. Note that the integrand has poles in $\H$ if
	\begin{equation*}
		x = x_m^{\pm,\varepsilon} := \frac{i}{\kappa} (\varepsilon km \mp \lambda) \in \mathbb{H} \Leftrightarrow \varepsilon k m \mp \lambda> 0, \quad \text{where $\varepsilon\in\{-1,1\}$}.
	\end{equation*}
	We compute, for $m\in\N$,
	\begin{equation*}
		\Res_{x= x_m^{\pm,\varepsilon}}\frac{\left(\frac{\pi\kappa}{k}x\pm\frac{\pi i\lambda}{k}\right)e^{2\pi ir\left(\frac{x}{z}\pm\frac{\ell}{\kappa}\right)}}{\left(\frac{\pi\kappa}{k}x\pm\frac{\pi i\lambda}{k}\right)^2+\pi^{2}m^2} = \frac{k}{2\pi \kappa} e^{ - \frac{2\pi r}{\kappa} \left( \frac{ \varepsilon km \mp \lambda }{z} \mp i \ell \right)}.
	\end{equation*}
	Now note that $x_m^{+,-1}\not\in\H$ and $x_m^{+,1}\in\H$. Using this, the part from $m\ge1$ and the $+$ contributes
	\begin{equation*}
		\frac{k}{2\kappa}\frac{e^{\frac{2\pi r}{\kappa} \left( \frac{\lambda}{z} + i \ell \right)}}{e^{\frac{2\pi rk}{\kappa z}}-1}.
	\end{equation*}

	Similarly the $-$ sign contributes
	\[
		\frac{k}{2\kappa}\frac{e^{-\frac{2\pi r}{\kappa} \left( \frac{\lambda}{z} + i \ell \right)}}{e^{\frac{2\pi rk}{\kappa z}}-1}.
	\]
	Adding these gives a contribution
	\begin{equation*}
		\frac{k}{\kappa} \frac{\cosh \left(\frac{2\pi r}{\kappa} \left( \frac{\lambda}{z} + i \ell \right)\right)}{e^{\frac{2\pi rk}{\kappa z}}-1}.
	\end{equation*}

	For $m=0$, we have only a contribution for the $-$ sign.
	Here we need to compute
	\[
		\Res_{x=\frac{i\lambda}\kappa}\frac{e^{2\pi ir\left(\frac{x}{z}-\frac{\ell}{\kappa}\right)}}{\frac{\pi\kappa}{k}x-\frac{\pi i \lambda}{k}}=\frac{ke^{-\frac{2\pi r}{\kappa}\left(\frac{\lambda}{z}+i\ell\right)}}{\pi\kappa}.
	\]
	This gives
	\[
		\frac1{2i}\sum_\pm \int_\R \frac{e^{2\pi ir\left(\frac xz\pm\frac\ell{\kappa}\right)}}{\frac{\pi\kappa}kx\pm\frac{\pi i \lambda}k } dx = \frac{k}{\kappa} e^{-\frac{2\pi r}{\kappa} \left( \frac{\lambda}{z} + i \ell \right)}.
	\]
	Hence
	\[
		F_{r,\ell,h,k}^{[2]}=  \frac{k}{2\kappa}e^{-\frac{2\pi r}{\kappa} \left( \frac{\lambda}{z} + i \ell \right)} + \frac{k\cosh \left(\frac{2\pi r}{\kappa} \left( \frac{\lambda}{z} + i \ell \right)\right)}{\kappa\left(e^{\frac{2\pi rk}{\kappa z}}-1\right)}.
	\]
	Combining gives the claim.
\end{proof}

We are now ready to determine the main asymptotic behavior of $L(q)$.

\begin{cor}\label{C:S2q}
	We have, as $z\to0$,
	\begin{multline*}
		L(q) =- \frac{\pi^2 \delta_{2\mid k}}{24 k^2 z^2} + \frac{\kappa}2 \sum_{\ell=1}^\kappa (-1)^\ell B_2\left(\frac\ell\kappa\right) \operatorname{Li}_2\left(\zeta_k^{h\ell}\right) + \frac{2\pi^2i\kappa^2}{3k} z \sum_{\ell=1}^\kappa (-1)^{\ell} \left\{\frac{h\ell}k\right\} B_3\left(\frac\ell\kappa\right)+ \frac{\pi^2}{8 k^2}z^2\\
		- \left(\frac{\kappa^2}{2\pi k}+\frac{3k}{2\pi} \delta_{2\nmid k}\right) \zeta(3)z + O\left(e^{-\frac{2\pi}{\kappa z}}\right).
	\end{multline*}
	In particular, if $2 \mid k$, then we have
	\[
		L(q) \ll \frac1{k^2 z^2} + k^2.
	\]
\end{cor}
\begin{proof}
	The first three terms are directly given in \Cref{L:S2Q}. For the remaining ones, we plug in \Cref{L:IntPM}. We evaluate
	\[
		\sum_{\ell=1}^{\kappa}(-1)^\ell\cos\left(\frac{2\pi r\ell}{\kappa}\right)=\frac{1}{2}\sum_{\ell=1}^{\kappa} (-1)^\ell\left(e^\frac{2\pi i r\ell}{\kappa}+e^{-\frac{2\pi i r\ell}{\kappa}}\right)=\frac{1}{2}\sum_{\pm} \sum_{\ell=1}^{\kappa}e^{\frac{2\pi i \ell}{\kappa}\left(\pm r+\frac{\kappa}{2}\right)}.
	\]
	The sum on $\ell$ now vanishes unless $r\equiv \mp\frac{\kappa}{2}\pmod{\kappa}$.
	Thus $\frac{\kappa}{2}\mid r$ and we change $r\mapsto\frac{\kappa}{2}r$. Then we need $r\equiv 1\pmod{2}$ and this part contributes
	\[
		\frac{\kappa^3}{2\pi k^2}\frac{1}{2\pi}\kappa\cdot3z^2\sum_{\substack{r\ge 1\\ 2\nmid r}}\frac{1}{\left(\frac{\kappa}{2}r\right)^4}=\frac{12z^2}{k^2\pi^2}\left(\sum_{r\ge 1}\frac{1}{r^4}-\sum_{\substack{r\ge 1\\ 2\mid r}}\frac{1}{r^4}\right)=\frac{12z^2}{k^2\pi^2}\left(1-\frac{1}{2^4}\right)\zeta(4)=\frac{\pi^2}{8k^2}z^2.
	\]
	This is the fourth term.

	The second term of Lemma \ref{L:IntPM} contributes
	\begin{equation}\label{second}
		-\frac{\kappa^2}{4\pi k} \frac{\del}{\del z}z^2\sum_{r\ge 1}\frac{1}{r^3}\sum_{\ell=1}^{\kappa}(-1)^\ell e^{-\frac{2\pi r}{\kappa}\left(\frac{\lambda}{z}+i\ell\right)}.
	\end{equation}
	To simplify this, recall that $\lambda$ is defined as $0\le \lambda< k$ with $\lambda\equiv h\ell\pmod{k}\Leftrightarrow \ell\equiv-\lambda h'\pmod{k}$. The goal now is to turn the sum on $\ell$ in a sum on $\lambda$.

	We first assume that $k$ is even. Then $\ell\equiv\lambda\pmod{2}$ and the sums on $\ell$ equals
	\[
		\sum_{\ell=1}^{k}(-1)^\lambda e^{-\frac{2\pi r}k\left(\frac{\lambda}{z}-ih'\lambda\right)} = \sum_{\lambda=0}^{k-1}\left(-e^{-\frac{2\pi r}k\left(\frac{1}{z}-ih'\right)}\right)^\lambda
	= \frac{1-e^{-2\pi r\left(\frac{1}{z}-ih'\right)}}{1+e^{-\frac{2\pi r}k\left(\frac{1}{z}-ih'\right)}} =\frac{1-e^{-\frac{2\pi r}{z}}}{1+\zeta_k^{h'r}e^{-\frac{2\pi r}{kz}}}.
	\]
	Thus, if $k$ is even, then \eqref{second} becomes
	\begin{multline*}
		-\frac k{4\pi} \sum_{r\ge1} \frac1{r^3} \frac\partial{\partial z} \frac{z^2\left(1-e^{-\frac{2\pi r}z}\right)}{1+\zeta_k^{h'r}e^{-\frac{2\pi r}{kz}}}
		= -\frac{kz}{2\pi} \sum_{r\ge1} \frac1{r^3} \frac{1-e^{-\frac{2\pi r}z}}{1+\zeta_k^{h'r}e^{-\frac{2\pi r}{kz}}} + \frac k2\sum_{r\ge1} \frac1{r^2} \frac{e^{-\frac{2\pi r}z}}{1+\zeta_k^{h'r}e^{-\frac{2\pi r}{kz}}}\\
		+ \frac12\sum_{r\ge1} \frac1{r^2} \frac{\left(1-e^{-\frac{2\pi r}z}\right)\zeta_k^{h'r}e^{-\frac{2\pi r}{kz}}}{\left(1+\zeta_k^{h'r}e^{-\frac{2\pi r}{kz}}\right)^2}
		= -\frac k{2\pi}\zeta(3)z + O\left(e^{-\frac{2\pi}{kz}}\right).
	\end{multline*}
	The first term above is the fifth term in Corollary \ref{C:S2q} if $k$ is even.

	Next assume that $k$ is odd. Then the sum on $\ell$ in \eqref{second} is
	\begin{equation}\label{E:SumEll}
		\sum_{\ell=1}^{2k}(-1)^\ell e^{-\frac{\pi r}k\left(\frac{\lambda}{z}+i\ell\right)}
		= \sum_{\ell=1}^{2k} (-1)^\ell \zeta_{2k}^{-r\ell} e^{-\frac{\pi r\lambda}{kz}}.
	\end{equation}
	If $r$ is even, then this sum vanishes as $\ell$ and $\ell+k$ lead to the same $\lambda$ and cancel due to the $(-1)^\ell$. If $r$ is odd, then we have by the Chinese Remainder Theorem
	\begin{equation}\label{CRT}
		(-1)^\ell \zeta_{2k}^{-r\ell} = \zeta_{2k}^{rh'\lambda}.
	\end{equation}
	Thus \eqref{E:SumEll} equals
	\begin{equation*}
		2\sum_{\lambda=0}^{k-1} \zeta_k^{r\frac{h'}2\lambda} e^{-\frac{\pi r\lambda}{kz}} = 2\frac{1-e^{-\frac{\pi r}z}}{1-\zeta_k^{r\frac{h'}2}e^{-\frac{\pi r}{kz}}}.
	\end{equation*}
	Thus if $k$ is odd, then (\ref{second}) becomes, as before
	\begin{align*}
		-\frac{2k}{\pi}\frac{\partial}{\partial z}z^2\sum_{\substack{r\ge1\\r\text{ odd}}}\frac{1}{r^3}\frac{1-e^{-\frac{\pi r}{z}}}{1-\zeta_k^{r\frac{h'}{2}}e^{-\frac{\pi r}{kz}}}
		= - \frac{7k}{2\pi}\zeta(3)z+O\left(e^{-\frac{\pi}{kz}}\right).
	\end{align*}
	The first term above gives the fifth term in Corollary \ref{C:S2q} if $k$ is odd.

	The final term from Lemma \ref{L:IntPM} contributes
	\begin{align}\label{final}
		-\frac{\kappa^2}{2\pi k} \sum_{\ell=1}^\kappa (-1)^\ell \sum_{r\ge1}\frac{1}{r^3}\frac{\del}{\del z}\frac{z^2\cosh\left(\frac{2\pi r}{\kappa}\left(\frac{\lambda}{z}+i\ell\right)\right)}{e^{\frac{2\pi kr}{\kappa z}}-1}.
	\end{align}
	We again want to write the sum on $\ell$ as a sum on $\lambda$ and distinguish as before whether $k$ is even or $k$ is odd.

	If $k$ is even, then the sum on $\ell$ becomes, again using that $\lambda\equiv h\ell\Pmod k$ and that $\ell\equiv\lambda\pmod{2}$,
	\begin{align*}
		\sum_{\ell=1}^k (-1)^\lambda \cosh\left(\frac{2\pi r}k\left(\frac{\lambda}{z}-ih'\lambda\right)\right) =\frac{1}{2}\sum_\pm \sum_{\lambda=0}^{k-1} (-1)^\lambda e^{\pm\frac{2\pi r\lambda}{k}\left(\frac{1}{z}-ih'\right)}
		=\frac{1}{2}\sum_\pm \frac{1-e^{\pm \frac{2\pi r}z}}{1+\zeta_k^{\mp h'r}e^{\pm\frac{2\pi r}{k z}}}.
	\end{align*}
	Thus, for $k$ even, \eqref{final} equals
	\begin{equation}\label{E:FinalEven}
		-\frac k{4\pi} \sum_\pm \sum_{r\ge1} \frac1{r^3} \frac\partial{\partial z} \frac{z^2\left(1-e^{\pm\frac{2\pi r}z}\right)}{\left(1+\zeta_k^{\mp h'r}e^{\pm\frac{2\pi r}{kz}}\right)\left(e^\frac{2\pi r}z-1\right)}.
	\end{equation}
	Next we combine
	\begin{align*}
		\frac{1-e^{\frac{2\pi r}z}}{\left(1+\zeta_k^{-h'r}e^{\frac{2\pi r}{kz}}\right)\left(e^{\frac{2\pi r}{z} }-1\right)}+\frac{1-e^{-\frac{2\pi r}{z}}}{\left(1+\zeta_k^{h'r}e^{-\frac{2\pi r}{kz}}\right)\left(e^{\frac{2\pi r}{z}}-1\right)}
		=\frac{e^{-\frac{2\pi r}{z}}-\zeta_k^{h' r}e^{-\frac{2\pi r}{kz}}}{1+\zeta_k^{h' r}e^{-\frac{2\pi r}{kz}}}.
	\end{align*}
	Thus \eqref{E:FinalEven} becomes
	\begin{align*}
		-\frac{k}{4\pi}\sum_{r\ge 1}\frac1{r^3}\frac{\del}{\del z}z^2 \frac{e^{-\frac{2\pi r}z}-\zeta_k^{h' r}e^{-\frac{2\pi r}{kz}}}{1+\zeta_k^{h' r}e^{-\frac{2\pi r}{kz}}} = O\left(e^{-\frac{2\pi}{kz}}\right).
	\end{align*}

	If $k$ is odd, then $\kappa = 2k$ and we have
	\begin{equation*}
		\sum_{\ell=1}^{2k} (-1)^\ell \cosh\left(\frac{\pi r}k\left(\frac\lambda z+i\ell\right)\right) = \frac12\sum_{\ell=1}^{2k} (-1)^\ell \left(\zeta_{2k}^{r\ell}e^\frac{\pi r\lambda}{kz}+\zeta_{2k}^{-r\ell}e^{-\frac{\pi r\lambda}{kz}}\right).
	\end{equation*}
	Now if $r$ is even, then the sum again vanishes
	as $\ell$ and $\ell+k$ lead to the same $\lambda$ but $(-1)^\ell$ introduces an extra minus.
	If $r$ is odd, then, using \eqref{CRT},
	\begin{align*}
		\frac12\sum_\pm\sum_{\ell=1}^{2k} (-1)^\ell \zeta_{2k}^{\pm r\ell} e^{\pm\frac{\pi r\lambda}{kz}} = \frac{\left(e^\frac{\pi r}z-1\right) \left(e^{-\frac{\pi r}z} + \zeta_{2k}^{rh'} e^{-\frac{\pi r}{kz}}\right)}{1 - \zeta_{2k}^{rh'} e^{-\frac{\pi r}{kz}} }.
	\end{align*}
	Thus, for $k$ odd, \eqref{final} becomes
	\begin{equation*}
		-\frac{2k}{\pi}\sum_{\substack{r\ge1\\r\text{ odd}}}\frac1{r^3} \frac{\del}{\del z}\frac{z^2\left(e^{-\frac{\pi r}z}+\zeta_{2k}^{rh'}e^{-\frac{\pi r}{kz}}\right)}{1-\zeta_{2k}^{rh'}e^{-\frac{\pi r}{kz}}} = O\left(e^{-\frac\pi{kz}}\right).
	\end{equation*}
	Combining errors yields the claim.

We next bound $L(q)$ for $2\mid k$ by estimating the individual terms. The first term is $O(\frac1{k^2z^2})$.
The second term may be bounded against
\[
	\ll k\sum_{\ell=1}^k \left|B_2\left(\frac{\ell}{k}\right)\right|\zeta(2) \ll k^2.
\]
The third term can be estimated against
\[
	\ll kz \sum_{\ell=1}^k \left|B_3\left(\frac \ell k\right)\right| \ll k^2 z \ll k.
\]
The fourth term is $O(\frac{z^2}{k^2})=O(1)$, the fifth term $O\left(kz\right)=O(1)$. Combining gives the claim.
\end{proof}

We define, for $2\nmid k$, which are constant term and the coefficient of $z$ in  Corollary \ref{C:S2q},
\begin{equation*}
	\alpha_{h,k} := k\sum_{\ell=1}^{2k} (-1)^\ell B_2\left(\frac\ell{2k}\right) \operatorname{Li}_2\left(\zeta_k^{h\ell}\right), \qquad
	\beta_{h,k} := \frac{8\pi^2 ik}3 \sum_{\ell =1}^{2k} (-1)^{\ell} \left\{\frac{h\ell}k\right\} B_3\left(\frac\ell{2k}\right) - \frac{7k}{2\pi} \zeta(3).
\end{equation*}
The following lemma simplifies $\alpha_{h,k}$ and $\beta_{h,k}$.

\begin{lem}\label{lem:a1a2eval}
For $k,h\in\N_0$, $k$ odd with $\gcd(h,k)=1$, we have
\begin{align*}
	\alpha_{h,k} &= \frac{(3k-1)\pi^2}{48} + \pi^2k \sum_{\ell =1}^{2k} (-1)^\ell \left(\!\!\left(\frac{h\ell}k\right)\!\!\right)^2  B_2  \left(\frac\ell{2k}\right),\\
	\beta_{h,k} &= \frac{8\pi^2ik}3 \sum_{\ell =1}^{2k} (-1)^{\ell} \left(\!\!\left(\frac{h\ell}k\right)\!\!\right) B_3 \left(\frac\ell{2k}\right)-\frac{7k}{2\pi}\zeta(3).
\end{align*}
\end{lem}
\begin{proof}
We first consider $\alpha_{h,k}$. Making the change of variables $\ell\mapsto2k-\ell$ in half of the sum and using \eqref{E:Bvan} gives that
\begin{equation}\label{E:a1compute}
	\sum_{\ell=1}^{2k} (-1)^\ell B_2\left(\frac\ell{2k}\right) \operatorname{Li}_2\left(\zeta_k^{h\ell}\right) = \frac12\sum_{\ell=1}^{2k} (-1)^\ell B_2\left(\frac\ell{2k}\right) \left(\operatorname{Li}_2\left(\zeta_k^{h\ell}\right)+\operatorname{Li}_2\left(\zeta_k^{-h\ell}\right)\right).
\end{equation}
The contribution from $k\mid\ell$ to \eqref{E:a1compute} is $\frac{\pi^2}{24}$.

For $k \nmid \ell$, we use \eqref{E:LRel} to obtain that the contribution to \eqref{E:a1compute} equals
\begin{equation}\label{E:Contribution}
	-\frac{\pi^2}{12} \sum_{\substack{\ell = 1\\ k \nmid \ell}}^{2k} (-1)^\ell B_2\left(\frac\ell{2k}\right) - \frac14 \sum_{\substack{\ell = 1\\ k \nmid \ell}}^{2k} (-1)^\ell B_2\left(\frac\ell{2k}\right) \Log^2\left(-\zeta_k^{h\ell}\right).
\end{equation}

We split the first sum into $\ell$ even and $\ell$ odd. If $\ell$ is even, then we write $\ell=2\nu$ with $1\le\nu\le k-1$. For $\ell$ odd, we write $\ell=2\nu-1$ with $1\le\nu\le k$. In this case, we need to skip $\nu = \frac{k+1}2$ which we add back in. This gives the contribution to \eqref{E:Contribution}
\begin{align*}
	-\frac{\pi^2}{12} \left( \sum_{\nu =1}^{k-1} B_2\left(\frac\nu k\right) - \sum_{\nu=1}^{k} B_2\left(\frac\nu  k-\frac1{2k}\right) + B_2\left(\frac{1}{2}\right) \right) = \frac{(k-1) \pi^2}{48 k}.
\end{align*}

For the second summand in \eqref{E:Contribution}, we write $-\zeta_k^{h\ell}=e^{2\pi i(\!(\frac{h\ell}k)\!)}$ using that $\frac{h\ell}k\not\in\IZ$. Because $(\!(\frac{h\ell}k)\!)\in(-\frac12,\frac12)$, we now have $\Log(-\zeta_k^{h\ell})=2\pi i(\!(\frac{h\ell}k)\!)$. Since $(\!(\frac{h\ell}k)\!)=0$ if $k\mid\ell$, we can add the terms with $k\mid\ell$ back in to obtain for the second summand in \eqref{E:Contribution}
\begin{equation*}
	\pi^2\sum_{\ell =1}^{2k} (-1)^\ell \left(\!\!\left(\frac{h\ell}k\right)\!\!\right)^2 B_2\left(\frac\ell{2k}\right).
\end{equation*}
Combining gives the claim for $\alpha_{h,k}$.

To evaluate $\beta_{h,k}$, we use that for $k\nmid\ell$, we have $\{\frac{h\ell}{k}\} = (\!(\frac{h\ell}{k})\!)+\frac12$. Thus, as $\sum_{\ell=1}^{2k} (-1)^\ell B_3\left( \frac{\ell}{2k}\right) = 0$, we conclude the claimed evaluation of $\beta_{h,k}$.
\end{proof}

Combining our investigations, we can now determine the asymptotic behavior of $S_2(q)$ towards roots of unity. For this, let
\begin{align*}
	a_{h,k} &:= \Log\left(\frac{\omega_{h,k}}{2\omega_{2h,k}^2}\right),\\
	b_{h,k} &:= a_{h,k}^2 + \frac{(3k-1)\pi^2}{48} + \pi^2k \sum_{\ell =1}^{2k} (-1)^\ell \left(\!\!\left(\frac{h\ell}k\right)\!\!\right)^2 B_2\left(\frac\ell{2k}\right) ,\\
	c_{h,k} &:= \frac\pi{2k}a_{h,k} + \frac{8\pi^2 ik}3 \sum_{\ell =1}^{2k} (-1)^{\ell} \left(\!\!\left(\frac{h\ell}k\right)\!\!\right) B_3\left(\frac\ell{2k}\right) - \frac{7k}{2\pi} \zeta(3).
\end{align*}
Using \Cref{L:CircleMethod}, \eqref{P2}, \eqref{Pn2}, and \Cref{C:S2q} gives the following behavior of $S_2 (q)$ as $z \to 0$.
\begin{cor}\label{C:Combine}\hspace{0cm}
	Let $z\to0$.
	\begin{enumerate}[leftmargin=*]
		\item[\rm(1)]If $2\mid k$, then
		\begin{equation*}
			S_2(q) \ll \left(\Log^2(z)+\frac1{k^2z^2}+k^2\right)e^{-\frac\pi{12kz}}.
		\end{equation*}
		\item[\rm(2)]If $2\nmid k$, then
		\begin{align*}
			S_2(q) = \frac{\omega_{h,k}}{\sqrt2\omega_{2h,k}} e^{\frac{\pi z}{12k}+\frac\pi{24kz}} \left(b_{h,k}+c_{h,k}z+\frac{3\pi^2}{16k^2}z^2-a_{h,k}\Log(z)-\frac{\pi z}{4k}\Log(z)+\frac{\Log^2(z)}4\right)\\
			+ O\left(\Log^2(z)e^{-\frac{23\pi}{24kz}}\right).
		\end{align*}
	\end{enumerate}
\end{cor}

\section{Proof of \Cref{T:MainIntro} and \Cref{T:SS}}\label{S:CM}

\subsection{The Circle Method set-up}

We use Cauchy's integral formula to obtain
\begin{equation}\label{E:SInt}
	s_j(n) =\frac{1}{2\pi i} \int_{|q|=e^{-\frac{2\pi}{n}}} \frac{S_j(q)}{q^{n+1}} dq= \sum_{\substack{0\le h<k\le N\\\gcd(h,k)=1}} e^{-\frac{2\pi i n h}k} \int_{-\vartheta_{h,k}'}^{\vartheta_{h,k}''} S_j\left(e^{\frac{2\pi i}k(h+iz)}\right) e^\frac{2\pi n z}k d\Phi,
\end{equation}
with the notation as in Subsection 2.6, with $N:=\lfloor\sqrt{n}\rfloor$. We now split \eqref{E:SInt}
\begin{equation*}
	s_j(n) = \sum\nolimits_{1}^{[j]} + \sum\nolimits_{2}^{[j]},
\end{equation*}
where
\begin{align*}
	\sum\nolimits_{1}^{[j]} &:= \sum_{\substack{0\le h<k\le N\\\gcd(h,k)=1\\2\mid k}} e^{-\frac{2\pi i n h}k} \int_{-\vartheta_{h,k}'}^{\vartheta_{h,k}''} S_j\left(e^{\frac{2\pi i}k(h+iz)}\right) e^\frac{2\pi n z}k d\Phi,\\
	\sum\nolimits_{2}^{[j]} &:= \sum_{\substack{0\le h<k\le N\\\gcd(h,k)=1\\2\nmid k}} e^{-\frac{2\pi i n h}k} \int_{-\vartheta_{h,k}'}^{\vartheta_{h,k}''} S_j \left(e^{\frac{2\pi i}k(h+iz)}\right) e^\frac{2\pi n z}k d\Phi.
\end{align*}

\subsection{Proof of \Cref{T:MainIntro}}
In this subsection we prove \Cref{T:MainIntro}.
\begin{proof}[Proof of \Cref{T:MainIntro}]
	Using \Cref{L:Main} (1), \eqref{Fbound}, and \eqref{zbound}, we may bound $\sum\nolimits_{1}^{[1]} \ll \sqrt n$.

	We next split off the principal part of $\sum\nolimits_{2}^{[1]}$ and write it as
	\begin{equation*}
		\sum\nolimits_{2}^{[1]} = \sum\nolimits_{21}^{[1]} + \sum\nolimits_{22}^{[1]},
	\end{equation*}
	where by \Cref{L:Main} (2)
	\begin{align*}
		\sum\nolimits_{21}^{[1]}&:= \frac{1}{\sqrt{2}} \sum_{\substack{0\le h<k\le N\\\gcd(h,k)=1\\2\nmid k}} \frac{\omega_{h,k}}{\omega_{2h,k}} e^{-\frac{2\pi inh}{k}}  \int_{-\vartheta_{h,k}'}^{\vartheta_{h,k}''} \!\left(\Log\!\left(\frac{\omega_{h,k}}{2\omega_{2h,k}^2}\right) - \frac{\Log(z)}2 + \frac{\pi z}{4k}\right) e^{\frac{2\pi}{k}\left(n+\frac{1}{24}\right)z+\frac{\pi}{24k z}} d\Phi,\\
		\sum\nolimits_{22}^{[1]} &:= \sum_{\substack{0\le h<k\le N\\\gcd(h,k)=1\\2\nmid k}} e^{-\frac{2\pi inh}k} \int_{-\vartheta_{h,k}'}^{\vartheta_{h,k}''} \mathcal E_{h,k}(z) e^\frac{2\pi nz}k d\Phi.
	\end{align*}
	We bound $\sum_{22}^{[1]}\ll\log(n)$. 

	To approximate $\sum_{21}^{[1]}$, define for $s\in\R_0^+$, $\ell\in\N_0$
	\begin{equation}\label{Isl}
		\mathcal{I}_{s,\ell}(A,B) := \int_{-\vartheta_{h,k}'}^{\vartheta_{h,k}''} z^s \Log^\ell(z) e^{Az+\frac{B}{z}} d\Phi
		= \frac1{ik} \int_{\frac kn-\frac{ik}{k\left(k+k_2\right)}}^{\frac kn+\frac{ik}{k\left(k+k_1\right)}} z^s \Log^\ell(z) e^{Az+\frac Bz} dz,
	\end{equation}
	with $A:=\frac{2\pi}{k} (n+\frac{1}{24})$, $B:= \frac{\pi}{24k}$, fixed throughout this section. Then
	\begin{equation}\label{E:Sum21}
		\sum\nolimits_{21}^{[1]} \!=\! \frac1{\sqrt2} \!\sum_{\substack{0\le h<k\le N\\\gcd(h,k)=1\\2\nmid k}}\! \frac{\omega_{h,k}}{\omega_{2h,k}} e^{-\frac{2\pi in h}{k} } \!\left(\Log\!\left(\frac{\omega_{h,k}}{2\omega_{2h,k}^2}\right)\!\mathcal{I}_{0,0}(A,B)\!-\! \frac{\mathcal{I}_{0,1}(A,B)}2 \!+\! \frac\pi{4k}\mathcal{I}_{1,0}(A,B)\right)\!.\hspace{-0.1cm}
	\end{equation}
	We now employ \Cref{L:BesselBound} with $\vartheta_1=\frac1{k(k+k_2)}$, $\vartheta_2=\frac1{k(k+k_1)}$, $a=A$, $b=B$, \eqref{I_sym}, \eqref{E:IK}, and \eqref{Bbound}, to obtain, as $n\to \infty$, for $j\in\{0,1,2,3\}$,
		\begin{align}
			\mathcal I_{j,0}(A,B) &= \frac{2\pi}{k} \left(\frac{B}{A}\right)^\frac{j+1}2 I_{j+1}\left(2\sqrt {AB}\right) + O\left(\frac1{kn^\frac{j+1}{2}}\right),\label{Ij0} \\
			\mathcal I_{0,1}(A,B) &= - \frac{\pi}{kA} I_{0}\left(2\sqrt{AB}\right) + \frac\pi k\sqrt{\frac{B}{A}}\log\left(\frac{B}{A}\right)I_{1}\left(2\sqrt{AB}\right)+O\left(\frac{\log(n)}{k\sqrt{n}}\right).\label{I01}
		\end{align}
	Plugging this into \eqref{E:Sum21} then gives \Cref{T:MainIntro}.
\end{proof}

\subsection{Proof of \Cref{T:SS}}
In this subsection we prove \Cref{T:SS}.
\begin{proof}[Proof of \Cref{T:SS}]
	Using \Cref{C:Combine} (1), \eqref{Fbound}, and \eqref{zbound}, we bound $\sum\nolimits_{1}^{[2]} \ll n^\frac32$.

	We next split off the principal part of $\sum\nolimits_{2}$ and write it as
	\begin{equation*}
		\sum\nolimits_{2}^{[2]} = \sum\nolimits_{21}^{[2]} + \sum\nolimits_{22}^{[2]},
	\end{equation*}
	where by \Cref{C:Combine} (2)
	\begin{align*}
		\sum\nolimits_{21}^{[2]}&:= \frac{1}{\sqrt{2}} \sum_{\substack{0\le h<k\le N\\\gcd(h,k)=1\\2\nmid k}} \frac{\omega_{h,k}}{\omega_{2h,k}} e^{-\frac{2\pi inh}{k}}\int_{-\vartheta_{h,k}'}^{\vartheta_{h,k}''} \bigg(b_{h,k}+c_{h,k}z+\frac{3\pi^2}{16k^2}z^2-a_{h,k}\Log(z)\\[-0.6cm]
		&\hspace{7cm} -\frac{\pi z}{4k}\Log(z)+\frac{\Log^2(z)}{4}\bigg) e^{\frac{2\pi}{k}\left(n+\frac{1}{24}\right)z+\frac{\pi}{24k z}} d\Phi,\\
		\sum\nolimits_{22}^{[2]} &= \frac1{\sqrt2}\sum_{\substack{0\le h<k\le N\\\gcd(h,k)=1\\2\nmid k}} \frac{\omega_{h,k}}{\omega_{2h,k}} e^{-\frac{2\pi inh}k} \int_{-\vartheta_{h,k}'}^{\vartheta_{h,k}''} \mathbb E_{h,k}(z)  e^\frac{2\pi nz}k d\Phi.
	\end{align*}
	Here
	\begin{equation*}
		\mathbb E_{h,k}(z) := S_2(q) - \left(b_{h,k}+c_{h,k}z+\frac{3\pi^2}{16k^2}z^2-a_{h,k}\Log(z)-\frac{\pi}{4k}z\Log(z)+\frac{\Log^2(z)}{4}\right)e^{\frac{\pi}{24kz}+\frac{\pi z}{12k}}.
	\end{equation*}
	Using \Cref{C:Combine} (2), \eqref{Fbound}, and \eqref{zbound}, we obtain
	\begin{align*}
		\sum\nolimits_{22}^{[2]}   \ll \log^2(n).
	\end{align*}

	We next write, using \eqref{Isl},
	\begin{multline}\label{E:S21}
		\sum\nolimits_{21}^{[2]}=\frac{1}{\sqrt{2}}\sum_{\substack{0\le h<k\le N\\\gcd(h,k)=1\\2\nmid k}}\frac{\omega_{h,k}}{\omega_{2h,k}}e^{-\frac{2\pi i n h}{k}}\bigg(b_{h,k}\mathcal{I}_{0,0}\left(A,B\right)+c_{h,k}\mathcal{I}_{1,0}\left(A,B\right)+\frac{3\pi^2}{16k^2}\mathcal{I}_{2,0}\left(A,B\right)\\[-0.6cm]
		-a_{h,k}\mathcal{I}_{0,1}\left(A,B\right)
		-\frac{\pi}{4k}\mathcal{I}_{1,1}\left(A,B\right)+\frac{\mathcal{I}_{0,2}\left(A,B\right)}{4} \bigg).
	\end{multline}

	 Using \Cref{L:BesselBound}, \eqref{I_sym}, and \eqref{E:IK}, we obtain
	\begin{align}\label{apI11}
		\mathcal{I}_{1,1}(A,B) &\!=\! \frac{\pi B}{kA} \log\!\left(\frac{B}{A}\right) \!I_{2}\!\left(2\sqrt{AB}\right) \!\!+\! \frac{\pi}{kA^2} I_{0}\!\left(2\sqrt{AB}\right) \!-\! \frac{2\pi\sqrt{B}}{kA^{\frac{3}{2}}} I_{1}\!\left(2\sqrt{AB}\right) \!\!+\! O\!\left(\frac{\log(n)}{kn}\right)\!.\!\!
	\end{align}
	We next show that, as $n\to\infty$,
	\begin{align}
		\mathcal{I}_{0,2}(A,B) = \frac{\pi \sqrt{B}}{2k\sqrt{A}} \log^2\left(\frac{A}{B}\right) I_{1}\left(2\sqrt{AB}\right) + \frac{\pi}{kA} \log\left( \frac{A}{B}\right)I_{0}\left(2\sqrt{AB}\right) -\frac{2\sqrt{B}}{k\sqrt{A}}\mathbb{I}\left(2\sqrt{AB}\right)&\nonumber\\
		&\hspace{-3.2cm}+ O\left(\frac{\log^2(n)}{k\sqrt{n}}\right)\!.\label{I02as}
	\end{align}
	For this, we compute, using \eqref{Iint}, as $x\to\infty$
	\begin{equation*}
		\left[\frac{\partial^2}{\partial \nu^2} I_\nu(x)\right]_{\nu=-1} = -\frac{\mathbb{I}(x)}\pi + O\left(e^{-\frac{x}2}\right).
	\end{equation*}
	Using \Cref{L:BesselBound}, \eqref{I_sym}, and \eqref{E:IK} gives
	\begin{multline*}
		\mathcal{I}_{0,2}(A,B)
		= \frac{\pi \sqrt{B}}{2k\sqrt{A}} \log^2\left(\frac{A}{B}\right) I_{1}\left(2\sqrt{AB}\right) + \frac{\pi}{kA} \log \left( \frac{A}{B} \right) I_{0}\left(2\sqrt{AB}\right) \\
		+ \frac{2\pi\sqrt{B}}{k\sqrt{A}}\left[\frac{\partial^2}{\partial \nu^2} I_\nu \left(2\sqrt{AB}\right)\right]_{\nu=-1} + \frac{2 \pi \sqrt{B}}{k\sqrt{A}} \log \left( \frac{A}{B} \right) K_{1}\left(2\sqrt{AB}\right) + O\left(\frac{\log^2(n)}{k\sqrt{n}}\right).
	\end{multline*}
	Using \eqref{Bbound} then yields \eqref{I02as}.

	Plugging \eqref{Ij0}, \eqref{I01}, \eqref{apI11}, and \eqref{I02as} into \eqref{E:S21} gives

	\begin{multline*}
		\sum\nolimits_{21}^{[2]}
		= \sum_{\substack{0\leq h<k\le N\\\gcd(h,k)=1\\2\nmid k}} \Bigg(\sum_{r=0}^3\gamma_r(h,k;n)I_r\left(X_k(n)\right) + \log(2(24n+1)) \sum_{j=0}^2\delta_j(h,k;n) I_j\left(X_k(n)\right)\\[-0.6cm]
		+ \log^2(2(24n+1)) \varrho(h,k;n) I_1(X_k(n)) + \psi(h,k;n) \mathbb I(X_k(n))\Bigg) + \mathcal{E},
	\end{multline*}
	where 
	\begin{align*} 
		\gamma_0(h,k;n) &:=  \frac{6\sqrt2\omega_{h,k}}{\omega_{2h,k}} \frac{e^{-\frac{2\pi inh}k}}{(24n+1)^2} (a_{h,k}(24n+1)-3), \\
		\gamma_1(h,k;n) &:= \frac{\pi\omega_{h,k}}{k\omega_{2h,k}} \frac{e^{-\frac{2\pi inh}k}}{(24n+1)^\frac32} (b_{h,k}(24n+1)+3), \\
		\gamma_2(h,k;n) &:= \frac{\pi\omega_{h,k}}{\sqrt2k\omega_{2h,k}} \frac{e^{-\frac{2\pi inh}k}}{24n+1}c_{h,k}, \quad
		\gamma_3(h,k;n) := \frac{3\pi^3\omega_{h,k}}{32k^3\omega_{2h,k}} \frac{e^{-\frac{2\pi inh}k}}{(24n+1)^\frac32},\\
		\delta_0(h,k;n) &:= \frac{3\omega_{h,k}}{\sqrt2\omega_{2h,k}} \frac{e^{-\frac{2\pi inh}k}}{24n+1}, \quad
		\delta_1(h,k;n) := \frac{\pi\omega_{h,k}}{2k\omega_{2h,k}} \frac{e^{-\frac{2\pi inh}k}}{\sqrt{24n+1}} a_{h,k}, \\
		\delta_2(h,k;n) &:= \frac{\pi^2\omega_{h,k}}{8\sqrt2k^2\omega_{2h,k}} \frac{e^{-\frac{2\pi inh}k}}{24n+1}, \quad
		\varrho(h,k;n) := \frac{\pi\omega_{h,k}}{16k\omega_{2h,k}} \frac{e^{-\frac{2\pi inh}k}}{\sqrt{24n+1}}, \\ \psi(h,k;n) &:= -\frac{\omega_{h,k}}{4k\omega_{2h,k}} \frac{e^{-\frac{2\pi inh}k}}{\sqrt{24n+1}},
	\end{align*}
	and where $\mathcal E$ satisfies
	\begin{equation*}
		\mathcal{E} \ll \sum_{\substack{0 \leq h < k \leq N \\ \gcd(h,k) = 1 \\ 2 \nmid k}} \frac{1}{k} \left( \frac{b_{h,k}}{\sqrt{n}} + \frac{c_{h,k}}{n} + \frac1{k^2n^{\frac{3}{2}}}  + \frac{a_{h,k}\log(n)}{\sqrt{n}} + \frac{\log(n)}{k n} + \frac{\log^2(n)}{\sqrt{n}}\right).
	\end{equation*}
	To further bound $\mathcal{E}$, we note that we have, directly from the definitions,
	\begin{equation*}
		a_{h,k} \ll 1, \quad b_{h,k} \ll k^2, \quad c_{h,k} \ll k^2.
	\end{equation*}
	Using $k \ll \sqrt n$ we then bound $\mathcal E \ll n$. Combining gives the theorem.
\end{proof}

\section{Proof of \Cref{C:SAsymp} and \Cref{C:SSAsymp}}\label{S:Proofs}

\subsection{Proof of \Cref{C:SAsymp}}
We are now ready to prove \Cref{C:SAsymp}.
\begin{proof}[Proof of \Cref{C:SAsymp}]
	Using \Cref{T:MainIntro}, \eqref{Bbound}, and noting that the main exponential term comes from $k=1$, the claim follows by expanding for, $r\in\R^+$,
	\begin{align}
		e^{\frac{\pi\sqrt{24n+1}}{6\sqrt2}} &= e^{\pi\sqrt{\frac n3}} \left(1 + \frac{\pi}{48\sqrt{3n}}+O\left(\frac1n\right)\right),\label{E:TExp}\\
		\frac1{(24n+1)^r}&=\frac1{(24n)^r}\left(1+O\left(\frac{1}{n}\right)\right), \quad  \frac{\log(2(24n+1))}4-\log(2) = \frac{\log(3n)}4 + O\left(\frac1n\right).\label{E:LogId}
	\end{align}
	This implies the claim.
\end{proof}

\subsection{Proof of \Cref{C:SSAsymp}}
Now we proceed to the proof of \Cref{C:SSAsymp}.
\begin{proof}[Proof of \Cref{C:SSAsymp}]
	Again the main exponential term comes from $k=1$. We drop all dependencies on $k$ and $h,$ and denote $X(n):=X_1(n)$, $\gamma_r(n):=\gamma_r(0,1;n)$ ($0\le r\le3$), $\delta_j(n):=\delta_j(0,1;n)$ ($0\le j\le2$), $\varrho(n):=\varrho(0,1;n)$, and $\psi(n):=\psi(0,1;n)$. Note that $X(n)\asymp \sqrt n$ and
	\begin{equation*}
		\gamma_0(n) \asymp \gamma_2(n) \asymp \delta_0(n) \asymp \delta_2(n) \asymp \frac1n,\qquad
		\gamma_1(n) \asymp \delta_1(n) \asymp \varrho(n) \asymp \psi(n) \asymp \frac1{\sqrt n}, \qquad
		\gamma_3(n) \asymp \frac1{n^\frac32}.
	\end{equation*}
	From \Cref{T:SS}, we have, using the first identity of \eqref{Bbound} and \Cref{L:BesselInt},
	\begin{align*}
		s_2(n) e^{-X(n)} \hspace{14.3cm}\\
		=\! \frac{2^{\frac14} \sqrt{3}}{\pi (24n\!+\!1)^{\!\frac14\!}} \!\left(\hspace{-0.03cm}\gamma_0(n)\!+\!\gamma_1\hspace{-0.03cm}(n)\!+\!\gamma_2(n)\!+\!\log(2(24n+1))\!\left(\delta_0(n)\!+\!\delta_1\hspace{-0.03cm}(n)\!+\!\delta_2(n)\right) \!+\! \log^2(2(24n\!+\!1)) \varrho(n)\hspace{-0.03cm}\right)\hspace{0.3cm}\\
		- \frac{3^{\frac52}}{2^{\frac54} {\pi^2} (24n\!+\!1)^{\!\frac34\!}}\!\left( \gamma_1\hspace{-0.03cm}(n) \!+\! \log(2(24n+1))\delta_1\hspace{-0.03cm}(n) \!+\! \log^2(2(24n+1))\varrho(n) \!-\! \frac{8\pi}{3} \psi(n) \right)\!+\! O\!\left(\frac{\log^2(n)}{n^\frac74}\right)\!.\hspace{0.3cm}
	\end{align*}
	We compute 
	\begin{equation*}
		a_{0,1} = -\log(2), \quad b_{0,1} = \log^2(2) + \frac{\pi^2}{24}, \quad \text{ and }\quad c_{0,1} = -\frac{\pi\log(2)}2 - \frac{7 \zeta(3)}{2\pi}.
	\end{equation*}
	This gives
	\begin{align*}
	\gamma_0(n) &= -\frac{6\sqrt{2} \log(2)}{24n+1} - \frac{18 \sqrt{2}}{(24n+1)^2}, \quad \gamma_1(n) = \frac{\pi \log^2(2)}{\sqrt{24n+1}}  +\frac{\pi^3}{24 \sqrt{24n+1}}  +\frac{3\pi}{(24n+1)^{\frac32}},\\
	\gamma_2(n) &= -\frac{\pi^2 \log(2)}{2\sqrt{2} (24n+1)} -\frac{7 \zeta(3)}{2\sqrt{2} (24n+1)}, \quad \delta_0(n) = \frac{3}{\sqrt{2} (24n+1)},\\
	\delta_1(n) &= -\frac{\pi \log(2)}{2 \sqrt{24n+1}}, \quad \delta_2(n) = \frac{\pi^2}{8\sqrt{2} (24n+1)}, \quad \varrho(n) = \frac{\pi}{16\sqrt{24n+1}}, \quad \psi(n) = - \frac{1}{4 \sqrt{24n+1}}.
	\end{align*}
	Hence we obtain 
	\begin{align*}
		s_2(n) e^{-X(n)} &= \frac{2^{\frac14} \sqrt{3}}{(24n+1)^{\frac34}} \Bigg( \left( \frac{\log(2(24n+1))}{4} -\log(2) \right)^2 + \frac{\pi^2}{24}  \Bigg)\\
		&\quad- \frac{3^{\frac52}}{2^{\frac54} {\pi} (24n+1)^{\frac54}} \Bigg( \left( \frac{\log(2(24n+1))}{4} -\log(2) \right)^2   + \frac{\pi^2}{24} + \frac23 \\
		&\hspace{61pt}-\frac{8}{3} \left(1+\frac{\pi^2}{24}\right) \left( \frac{\log(2(24n+1))}{4} -\log(2) \right)  + \frac{7 \zeta(3)}{9} \Bigg)  + O\!\left(\frac{\log^2(n)}{n^\frac74}\right).
	\end{align*} 
	\Cref{C:SSAsymp} then follows using \eqref{E:TExp} and \eqref{E:LogId}.
\end{proof}

\section{Outlook}\label{S:Outlook}
We finish the paper with some open questions. First it is interesting to study higher moments. For this, let $s_k(n)$ be the {\it$k$-th moment of the reciprocal sums of partitions} for $k\in\N_0$, i.e.,
\[
	s_k(n) := \sum_{\lambda \in \mathcal{D}_n} \left( \sum_{j=1}^{\ell(\lambda)} \frac{1}{\lambda_j}\right)^k.
\]
The generating function for $s_k(n)$ is
\[
	\sum_{n \geq 1} s_k (n) q^n
	= \left[ \left( \zeta \frac{d}{d\zeta} \right)^k \sum_{\lambda \in \mathcal{D}} q^{|\lambda|} \zeta^{\sum_{j=1}^{\ell(\lambda)} \frac1{\lambda_j} } \right]_{\zeta=1}
	= \left[ \left(\zeta \frac{d}{d\zeta} \right)^k \prod_{r\ge1} \left( 1+ \zeta^\frac1rq^r \right) \right]_{\zeta=1}.
\]
Now we define
\[
	g_k (q) := \left[ \left(\zeta \frac{d}{d\zeta} \right)^k \sum_{r\ge1} \Log \left(1+\zeta^{\frac1r} q^r\right) \right]_{\zeta=1}.
\]
Note that $g_2 (q) = L(q)$ and that
\[
	\frac{1}{(-q;q)_{\infty}} \sum_{n \geq 1} s_k (n) q^n \in \mathbb{Z}\left[g_1(q), g_2(q), \ldots, g_k (q)\right].
\]
The later statement follows using
\[
	\zeta \frac{d}{d\zeta} \prod_{r\ge1} \left( 1+ \zeta^\frac1rq^r \right) = \prod_{r\ge1} \left( 1+ \zeta^\frac1rq^r \right)\zeta \frac{d}{d\zeta} \sum_{r\ge1} \Log \left(1+\zeta^{\frac1r} q^r\right)
\]
and induction. It is then natural to ask the asymptotic behavior of $g_k (q)$ and its coefficients.

Another interesting  question is whether $S_2(q)$ lies in some space of ``modular-like'' objects. This boils down to understanding modularity properties of $L(q)$. Lemmas \ref{L:S2Q} and \ref{L:IntPM} give the ``obstruction to modularity'' of $L(q)$. The question that then arises is whether these extra terms have period-like properties and whether $L(q)$ poses some ``completion'' that transforms like a modular form. In a similar manner, one can ask whether one can complete the generating function of $s_k(n)$.


\end{document}